\documentclass{article}

\usepackage{microtype}
\usepackage{graphicx}
\usepackage{subfigure}
\usepackage{booktabs} 

\usepackage{hyperref}

\usepackage[accepted]{icml2021}

\usepackage{amsfonts}
\usepackage{mathtools}
\usepackage{amsthm}
\usepackage{thm-restate}
\usepackage{amsmath}
\usepackage{mathrsfs}
\usepackage{natbib}
\usepackage{algorithm}
\usepackage{algorithmic}
\usepackage[inline]{enumitem}

\newcommand{\Rd}{\mathbb{R}^d}
\newcommand{\R}{\mathbb{R}}

\newcommand{\E}[1]{\mathbb{E} \left[ #1 \right]}
\newcommand{\Exp}[2]{\mathop{\mathbb{E}}_{#1} \left[ #2 \right]}

\newcommand{\sqnorm}[1]{\left\lVert#1\right\rVert_2^2}
\newcommand{\norm}[1]{\left\lVert#1\right\rVert_2}
\newcommand{\xstar}{x^{*}}
\newcommand{\eps}{\varepsilon}
\newcommand{\pkik}{p_k^{i_k}}
\newcommand{\suchthat}{\;\ifnum\currentgrouptype=16 \middle\fi|\;}
\newcommand{\pluseq}{\mathrel{+}=}
\newcommand{\minuseq}{\mathrel{-}=}

\DeclarePairedDelimiterX{\inp}[2]{\langle}{\rangle}{#1, #2}

\DeclareMathOperator*{\argmin}{arg\,min}
\DeclarePairedDelimiter\floor{\lfloor}{\rfloor}
\DeclarePairedDelimiter\abs{\lvert}{\rvert}%

\newtheorem{theorem}{Theorem}
\newtheorem{assumption}{Assumption}
\newtheorem{lemma}{Lemma}
\newtheorem{proposition}{Proposition}
\newtheorem{corollary}{Corollary}

\icmltitlerunning{Stochastic Reweighted Gradient Descent}

\begin{document}

\twocolumn[
\icmltitle{Stochastic Reweighted Gradient Descent}



\icmlsetsymbol{equal}{*}

\begin{icmlauthorlist}
\icmlauthor{Ayoub El Hanchi}{to}
\icmlauthor{David A. Stephens}{to}
\end{icmlauthorlist}

\icmlaffiliation{to}{Department of Mathematics and Statistics, 
McGill University, Montreal, Canada}

\icmlcorrespondingauthor{Ayoub El Hanchi}{ayoub.elhanchi@mail.mcgill.ca}

\icmlkeywords{Machine Learning, ICML}

\vskip 0.3in
]


\printAffiliationsAndNotice{}  

\begin{abstract}
    Despite the strong theoretical guarantees that variance-reduced finite-sum 
    optimization algorithms enjoy, their applicability remains limited to cases 
    where the memory overhead they introduce (SAG/SAGA), or the periodic full 
    gradient computation they require (SVRG/SARAH) are manageable.
    A promising approach to achieving variance reduction while avoiding these 
    drawbacks is the use of importance sampling instead of control variates. 
    While many such methods have been proposed in the literature,
    directly proving that they improve the convergence of the resulting optimization 
    algorithm has remained elusive.
    In this work, we propose an importance-sampling-based algorithm we call SRG
    (stochastic reweighted gradient).
    We analyze the convergence of SRG in the strongly-convex case and show that, while
    it does not recover the linear rate of control variates methods, it provably
    outperforms SGD.
    We pay particular attention to the time and memory overhead of our proposed method,
    and design a specialized red-black tree allowing its efficient
    implementation. Finally, we present empirical results to support our findings.
\end{abstract}

\section{Introduction}
\label{introduction}
We are interested in the unconstrained minimization problem:
\begin{equation}
    \min_{x \in \Rd} F(x) \coloneqq \Exp{\xi \sim P}{f(x, \xi)}
\label{stoch_problem}
\end{equation}
where $\xi$ is a real random vector with distribution $P$.
In machine learning, $x$ represents the parameters of the prediction rule,
$\xi$ represents the pair of input/output variables, 
$f$ the loss function, and $F$ the risk.

In particular, we focus on the sample average approximation to
(\ref{stoch_problem}):
\begin{equation}
    \min_{x \in \Rd} F_n(x) \coloneqq \frac{1}{n} \sum_{i=1}^{n} f(x, \xi_i)
\label{sample_avg_approx}
\end{equation}
where $(\xi_i)_{i=1}^{n}$ are i.i.d. samples from $P$.
This is also known as empirical risk minimization in machine learning.

Defining $f_i(x) \coloneqq f(x, \xi_i)$ for $i \in [n]$, and letting
$i$ be a uniformly distributed random variable on $[n]$, we can rewrite
(\ref{sample_avg_approx}) as:
\begin{equation}
    \min_{x \in \Rd} F(x) \coloneqq \frac{1}{n} \sum_{i=1}^{n} f_i(x)
    = \E{f_i(x)} 
    \label{problem}
\end{equation}
where the expectation is taken with respect to $i$.

We make the following assumptions on the function $F$ and the component
functions $f_i$:
\begin{assumption}
    The function $F: \Rd \to \R$ is differentiable and
    $\mu$-strongly convex, that is $\forall x,y \in \Rd$:
    \begin{equation*}
        F(y) \geq  F(x) + \inp{\nabla F(x)}{y - x} + \frac{\mu}{2}\sqnorm{y - x} 
    \end{equation*}
    \label{strongly-convex}
\end{assumption}
\begin{assumption}
    For all $i \in [n]$, the functions $f_i: \Rd \to \R$ are differentiable
    and convex, that is, $\forall x,y \in \Rd$:
    \begin{equation*}
        f_i(y) \geq f_i(x) + \inp{\nabla f_i(x)}{y - x}
    \end{equation*}
    \label{convex}
\end{assumption}
\begin{assumption}
    For all $i \in [n]$, the functions $f_i: \Rd \to \R$ are $L$-smooth,
    that is, $\forall x,y \in \Rd$:
    \begin{equation*}
        \norm{\nabla f_i(y) - \nabla f_i(x)} \leq L \norm{y - x}
    \end{equation*}
    \label{smooth}
\end{assumption}
Note that by Assumption \ref{smooth} and the triangle inequality, 
$F$ is also $L$-smooth, and we define its condition number by $\kappa = L/\mu$.

Let $\xstar$ be the unique minimizer of $F$.
Our goal is to design the fastest algorithm that outputs a point $x \in \Rd$
such that $\sqnorm{x - \xstar} < \eps$ for a given accuracy $\eps > 0$. 
Following \cite{agarwal_lower_2015}, we assume that we have access
to a gradient oracle that takes as input a point $x \in \Rd$ and an index $i \in [n]$
and returns $\nabla f_i(x)$. We then measure the complexity of an algorithm
by counting the number of oracle calls needed to achieve
a desired accuracy. 

The most straightforward way to solve (\ref{problem}) is to
ignore the particular structure of $F$, and simply run gradient descent
or accelerated gradient-descent \cite{nesterov_method_1983} on $F$.
Under our assumptions and definition of complexity,
gradient descent converges at the rate of $O(n\kappa\log{(1/\eps)})$,
while accelerated gradient descent converges at the rate of
$O(n\sqrt{\kappa}\log{(1/\eps)})$ \cite{nesterov_lectures_2018}. 
When $n$ is large, this becomes prohibitively expensive.

One solution to this is to view problem (\ref{problem}) in its expectation form,
and use the stochastic approximation of gradient descent.
This yields stochastic gradient descent (SGD)
which estimates $\nabla F(x)$ by $\nabla f_i(x)$ for $i$ 
uniformly distributed on $[n]$ in the gradient descent update
\cite{robbins_stochastic_1951,nemirovsky_problem_1983,nemirovski_robust_2009,
moulines_non-asymptotic_2011}. The complexity of SGD under our assumptions
is known to have two regimes. Denote by $\sigma^2$ the variance of the gradient estimator 
of SGD at the minimizer $\xstar$. For $\eps > \sigma^2/L\mu$, SGD converges
at the fast rate $O(\kappa \log{(1/\eps)})$, while for $\eps \leq \sigma^2/L\mu$,
SGD suffers from the sublinear rate $O(\sigma^2/\mu^2\varepsilon)$
\cite{needell_stochastic_2014,nguyen_sgd_2018,gower_sgd_2019}.
When $\sigma^2$ is large, SGD takes a long time
to converge to even moderately accurate solutions.

While SGD does take into account the expectation form of the objective
(\ref{problem}), it fails to adapt to the fact that the expectation is taken over
a discrete distribution which gives rise to a finite-sum structure.
The last decade has seen the development of many new methods that leverage
this additional structure, including (but not limited to) 
SAG \cite{roux_stochastic_2012,schmidt_minimizing_2017},
SAGA \cite{defazio_saga_2014},
SVRG \cite{johnson_accelerating_2013},
and SARAH \cite{nguyen_sarah_2017}. All of these methods
converge at the fast linear rate of $O((n + \kappa)\log{(1/\eps)})$.
Further work in this direction led to the development of lower-bounds 
for the optimization of finite-sum functions under our assumptions 
and the oracle model of complexity
\cite{agarwal_lower_2015,woodworth_tight_2016,arjevani_limitations_2017,lan_optimal_2018},
and, similar to the deterministic case,
accelerated methods have been developed that closely match them
\cite{lin_catalyst_2018,allen-zhu_katyusha_2018,lan_unified_2019,song_variance_2020}.
The core innovation behind this set of algorithms is the design of more efficient 
gradient estimators using carefully designed control variates, earning them the
name of variance-reduced methods.

Despite all this progress, SGD remains the algorithm
of choice in practice for some machine learning problems where it 
empirically outperforms variance-reduced methods \cite{defazio_ineffectiveness_2019}.
From an optimization perspective, there are a few explanations
to this observation.

One is the phenomenon of interpolation
\cite{ma_power_2018,vaswani_fast_2019,vaswani_painless_2019}:
if $\sigma^2 = 0$, then SGD converges linearly for any $\eps > 0$ without
any variance reduction. A second possible explanation is that
in many cases we are content with a low accuracy solution $\eps > \sigma^2/L\mu$,
in which case SGD converges linearly as well.
In either of these scenarios, SGD and variance-reduced methods have
the same iteration complexity, but variance-reduced methods require
on average three times as many gradient evaluations per iteration
(SVRG/SARAH).

While these reasons can be 
used to explain the superior performance of SGD in some settings,
there is evidence that there are scenarios where variance-reduction might be useful
even when only moderate accuracy is needed.
In particular, \cite{goyal_accurate_2018} showed that reducing the
variance of the gradient estimator by increasing the batch size
leads to optimization gains in the training 
of deep neural networks for ImageNet.
Even in such cases however, the user is faced with a new problem.
If we use variance-reduced methods from the beginning of the optimization,
we make progress three times more slowly than SGD in the initial phase,
which is often the dominating phase.
Ideally, we would like to pay the additional computational cost
for variance-reduction only when we need it, but it is not clear
how we can detect when progress becomes constrained by the variance
of the gradient estimator,
although some advances have been made in this direction, see for example
\cite{pesme_convergence-diagnostic_2020}.

A pressing question that arises from our discussion is therefore:
\textit{Can we develop a variance-reduced algorithm that uses exactly the same
number of gradient evaluations as SGD ?}
Such an algorithm would enjoy the fast rate of SGD in the initial phase
of optimization, while automatically performing variance-reduction when needed.
This question is the main motivation for this paper.

Before closing this section, we give one additional motivation for our work.
As previously mentioned, the core idea of variance-reduced methods
is the use of carefully crafted control variates to reduce the variance
of the gradient estimator. There are, however, other ways to reduce 
the variance of a Monte Carlo estimator. In particular,
one can use importance sampling.
There is a large literature on such methods.
Initial attempts focused on using a fixed distribution throughout
the optimization based on prior knowledge about the functions $f_i$
\cite{needell_stochastic_2014,zhao_stochastic_2015,csiba_importance_2018},
and were able to show slightly favorable rates compared to uniform sampling.
A more recent line of work attempts to adaptively design the distributions
\cite{schaul_prioritized_2015,loshchilov_online_2015,alain_variance_2016,bouchard_online_2016,canevet_importance_2016,stich_safe_2017,katharopoulos_not_2018,johnson_training_2018},
although the methods developed in these works mostly rely on heuristics 
and do not come with theoretical guarantees.
Finally, in an attempt to put adaptive importance sampling
methods on a firmer theoretical ground, a recent set of
papers embed the problem of designing the distributions 
in an online learning framework and propose methods
that achieve sublinear regret 
\cite{namkoong_adaptive_2017,salehi_stochastic_2017,borsos_online_2018,borsos_online_2019,el_hanchi_adaptive_2020}.
These analyses however are still not satisfactory as
they assume that the gradient norms of the functions $f_i$ are uniformly
bounded on $\Rd$ which does not hold, for example, 
when $F$ is strongly-convex.
The second question we ask here is therefore: 
\textit{Can we develop a variance-reduced
algorithm based on importance sampling and which provably outperforms SGD ?}

\textbf{Contributions}: In this paper, we give positive answers to both questions.
In particular, our contributions are as follows:
\begin{itemize}
    \item We propose stochastic reweighted gradient descent (SRG), a
    variance-reduced algorithm for the optimization of finite-sums
    based on importance sampling.
    \item As oppose to SAGA/SAG which require $O(nd)$ memory,
    SRG only requires $O(n)$ memory.
    \item Unlike SVRG/SARAH which require on average three gradient
    evaluations per iteration, and like SGD, SRG requires a 
    single gradient evaluation per iteration.
    \item We develop a specialized red-black tree that allows an efficient
    implementation of SRG, incurring an overhead of only $O(\log{n})$ operations
    per iteration compared to SGD.
    \item We show that SRG provably outperform SGD under our assumptions.
    Let $\sigma^2_{*}$ be the minimal variance achievable through importance
    sampling at the minimizer $\xstar$. Then we show that SRG has
    the same convergence as SGD, but with $\sigma^2$ replaced with
    $\sigma_*^2$, which can be up to $n$ times smaller.
\end{itemize}

\section{Algorithm}
\label{algorithm}
\begin{algorithm}[h]
    \caption{SRG}
    \begin{algorithmic}
        \STATE{\textbf{Parameters:}}
        step sizes $(\alpha_k)_{k=0}^{\infty} > 0$, 
        lower bounds $(\varepsilon_k)_{k=0}^{\infty} \in (0, \frac{1}{n}]$
        \STATE{\textbf{Initialization:}} $x_0 \in \Rd, (g_0^i)_{i=1}^{n} \in \Rd $
        \FOR{$k = 0, 1, 2, \dotsc$}
        {
            \STATE $p_k = \argmin_{p \in \Delta(\eps_k)} \frac{1}{n^2} \sum_{i=1}^{n} \frac{1}{p^i}\sqnorm{g_k^i}$
            \STATE sample $i_k \sim p_k$
            \STATE $x_{k+1} = x_k - \alpha_k \frac{1}{np_{k}^{i_k}} \nabla f_{i_k}(x_k)$
            \STATE sample $b_k \sim \text{Bernoulli}(\frac{\eps_k}{p_k^{i_k}})$
            \STATE $
            g_{k+1}^i =
            \begin{cases}
                \nabla f_i(x_k) & \text{if} \ i = i_k \ \text{and} \ b_k = 1\\
                g_{k}^i & \text{otherwise} \\
            \end{cases}
            $
        }
        \ENDFOR
    \end{algorithmic}
    \label{SRG}
\end{algorithm}

Before introducing our proposed algorithm, let us first introduce some notation.
For an iteration number $k \in \mathbb{N}$, SGD performs the following update
to minimize (\ref{problem}):
\begin{equation*}
    x_{k+1} = x_k - \alpha_k \nabla f_{i_k}(x_k)
    \label{SGD}
\end{equation*}
for a step size $\alpha_k > 0$ and a random index $i_k$ drawn uniformly from $[n]$ 
and independently at each iteration.
The idea behind importance sampling is to instead sample the index $i_k$ according
to a given distribution $p_k$ on $[n]$, and to perform the update:
\begin{equation*}
    x_{k+1} = x_k - \alpha_k \frac{1}{np_k^{i_k}} \nabla f_{i_k}(x_k)
\end{equation*}
where $p_k^{i}$ is the $i^{th}$ component of the vector $p_k \in \Delta$,
and $\Delta$ is the probability simplex in $\mathbb{R}^n$.
It is immediate to verify that the importance sampling estimator
of the gradient is unbiased as long as $p_k > 0$. The challenge
is to design a sequence $\{p_k\}_{k=0}^{\infty}$ that produces
more efficient gradient estimators than the ones produced by uniform sampling.

Our proposed method is a simple modification of the one given by
\cite{el_hanchi_adaptive_2020}, and is given in Algorithm \ref{SRG}.
The algorithm follows from the following reasoning.
The variance of the importance sampling gradient estimator is given by, up to
an additive constant:
\begin{align}
    \begin{split}
    \sigma^2(x_k, p_k) &\coloneqq \Exp{i_k \sim p_k}{\sqnorm{\frac{1}{np_k^{i_k}}\nabla f_i(x_k)}} \\
    &= \frac{1}{n^2}\sum_{i=1}^{n} \frac{1}{p_k^i} \sqnorm{\nabla f_i(x_k)}
    \end{split}
    \label{var}
\end{align}
Ideally we would like to choose $p_k \in \Delta$ so that this variance
is minimized, but this is not feasible as we do not have access
to the gradient norms $\norm{\nabla f_i(x_k)}$.
Inspired by control variates methods, and in particular SAG/SAGA, 
we instead maintain a table $(g_k^i)_{i=1}^{n}$ 
that aims to track the component gradients $(\nabla f_i(x_k))_{i=1}^{n}$. 
We then use this table to construct an approximation of the true variance:
\begin{equation}
    \tilde{\sigma}^2(x_k, p_k) \coloneqq \frac{1}{n^2} \sum_{i=1}^{n} \frac{1}{p^{i}_k} \sqnorm{g_k^i}
    \label{approx_var}
\end{equation}
Finally, we choose $p_k$ so as to minimize this quantity. Since this
is only an approximation, we perform the minimization over the restricted
simplex $\Delta(\eps_k) = \{p \in \Delta \mid p \geq \eps_k\}$ for
a given $\eps_k \in (0, 1/n]$. This enforces $p_k > 0$, while
making sure the error on the approximation of the variance is taken into
account. 
The difference between our method and the one proposed in 
\cite{el_hanchi_adaptive_2020} is in the way the table $(g_k^i)_{i=1}^{n}$
is updated: we add a bernoulli random variable to determine whether
the update occurs or not. This step is crucial for the analysis as we
will see, although in practice it might be skipped.

The goal of the next two sections will be to
\begin{enumerate*}[label=(\roman*)]
    \item show that SRG can be efficiently implemented.
    \item analyse the convergence of SRG and show that it outperforms SGD.
\end{enumerate*}


\section{Implementation}
\label{implementation}
In this section, we show how Algorithm \ref{SRG} can
be efficiently implemented, leading
to a per-iteration cost that is competitive with SGD.

First, note that while we use the table $(g_k^i)_{i=1}^{n}$
in the presentation of Algorithm \ref{SRG}, in practice 
it is enough to store the gradient norms 
$(\norm{g_k^i})_{i=1}^{n}$ only.
SRG therefore requires only $O(n)$ memory compared to SAG/SAGA
which require $O(nd)$ memory in general.

With memory issues out of the way, we turn to the main bottleneck
in SRG which is sampling from:
\begin{equation}
    p_k = \argmin_{p \in \Delta(\eps_k)} \frac{1}{n^2} \sum_{i=1}^{n} \frac{1}{p^i}
    \sqnorm{g_k^i}
    \label{srg_simplex_problem}
\end{equation}

The following lemma, taken from \cite{el_hanchi_adaptive_2020},
sheds some light on the solution:
\begin{lemma}
    Let $\{a_i\}_{i=1}^{n}$ be a non-negative set of numbers and
    assume that there is some $i \in [n]$ such that $a_i > 0$.
    Let $\eps \in [0, 1/n]$, and
    let $\pi: [n] \rightarrow [n]$ be a permutation that orders 
    $\{a_i\}_{i=1}^{N}$ in a decreasing order 
    ($a_{\pi(1)} \geq a_{\pi(2)} \geq \dots \geq a_{\pi(N)}$).
    Define for $i \in [n]$:
    \begin{equation}
        \label{lambda}
        \lambda(i) \coloneqq 
        \frac{\sum_{j=1}^{i} a_{\pi(j)}}{1-(n-i)\varepsilon}
    \end{equation}
    and:
    \begin{equation}
        \label{rho}
        \rho:= \max \left\{i \in [n] \suchthat a_{\pi(i)} \geq \eps \lambda(i)  \right\}
    \end{equation}
    Then a solution of the optimization problem:
    \begin{equation}
        p^{*} = \argmin_{p \in \Delta(\varepsilon)} \sum_{i=1}^{n} \frac{1}{p_i} a_i^2
        \label{simplex_problem}
    \end{equation}
    is given by:
    \begin{equation}
    \label{probs}
        p^{*}_i = \begin{cases}
        a_i/\lambda(\rho) &\text{if $\pi(i) \leq \rho$}\\
        \varepsilon &\text{otherwise}
        \end{cases}
    \end{equation}
    In the case $a_i = 0$ for all $i \in [n]$, any 
    $p \in \Delta(\eps)$ is a solution,
    and in the case $\eps = 1/n$, $p = 1/n$ is the unique solution.
    \label{solution}
\end{lemma}

\subsection{Naive implementation}

\begin{algorithm}[htbp]
    \caption{Naive sample}
    \begin{algorithmic}
        \STATE{\textbf{Input:} 
        $(a_{\pi(i)})_{i=1}^{n} > 0, 
        (\lambda(i))_{i=1}^{n},
        \eps > 0$
        }

        \STATE{\textbf{Initialization:}} 
        $l \gets 1, r \gets n$

        \WHILE{$l \leq r$}
            \STATE $m \gets \floor{\frac{l + r}{2}}$
            \IF{$a_{\pi(m)} \geq \eps \lambda(m)$}
                \STATE{$c \gets m$}
                \STATE{$l \gets m + 1$}
            \ELSE
                \STATE{$r \gets m - 1$}
            \ENDIF
        \ENDWHILE

        \STATE{$\rho \gets c$}
        \STATE{compute $p^{*}$ using (\ref{probs}).}
        \STATE{sample $i \sim p^{*}$.}
        \STATE{\textbf{Output:} the index $i$ and its probability $p^{*, i}$.}
    \end{algorithmic}
    \label{naive_sample}
\end{algorithm}

To simplify notation for this section we will use the one
introduced in Lemma \ref{solution}, that is,
we will refer to $a_i$ instead of $\norm{g_k^i}$.

Let us for the moment assume that we have access
to $(a_{\pi(i)})_{i=1}^{n}$, a sorted version of $(a_{i})_{i=1}^{n}$,
as well as to $(\lambda(i))_{i=1}^{n}$. We will worry
about how to maintain these in the next subsection.
How can we sample from (\ref{simplex_problem}) ?

The following proposition reveals a useful property, which
we use to construct Algorithm \ref{naive_sample}:
\begin{proposition}
    With the definitions in Lemma \ref{solution}, we have:
    \begin{align*}
        a_{\pi(i)} \geq \eps\lambda(i)
        &\Rightarrow a_{\pi(l)} \geq \eps\lambda(l) \quad \forall l \leq i \\
        a_{\pi(i)} < \eps\lambda(i)
        &\Rightarrow a_{\pi(l)} < \eps\lambda(l) \quad \forall l \geq i
    \end{align*}
    \label{order}
\end{proposition}
\begin{proof}
    The proof of the first and second statements are the same
    (replacing $<$ with $\geq$).
    We prove here the second statement.
    Let $i \in [n]$ such that $a_{\pi(i)} < \eps\lambda(i)$.
    By definition of $\pi$, $a_{\pi(i+1)} \leq a_{\pi(i)}$, so we have:
    \begin{align*}
        &a_{\pi(i+1)} < \eps \lambda(i) \\
        \Leftrightarrow \quad
        & a_{\pi(i+1)} (1 - (n-i)\eps) < \eps \sum_{j=1}^{i} a_{\pi(j)} \\
        \Leftrightarrow \quad
        & a_{\pi(i+1)} (1 - (n-(i+1))\eps) < \eps \sum_{j=1}^{i+1} a_{\pi(j)} \\
        \Leftrightarrow \quad
        & a_{\pi(i+1)} < \eps \lambda(i+1)
    \end{align*}
    the rest of the claim holds by induction.
\end{proof}

\begin{lemma}
    With the definitions of Lemma \ref{solution}, Algorithm 
    \ref{naive_sample} samples from the distribution
    given by (\ref{simplex_problem}).
\end{lemma}
\begin{proof}
    Algorithm \ref{naive_sample} proceeds by first finding
    $\rho$, then uses Lemma \ref{solution} to construct and
    sample from $p^{*}$. It is therefore enough to show
    that it indeed finds $\rho$. First, note that from the proof 
    of Lemma \ref{solution} in \cite{el_hanchi_adaptive_2020}, 
    we know that:
    \begin{equation*}
        a_{\pi(1)} \geq \eps \lambda(1)
    \end{equation*}
    therefore the variable $c$ is guaranteed to
    be well defined by the end of the $while$ loop.
    We claim that the $while$ loop maintains the following invariant:
    $\rho \in \{c\} \cup [l, r]$ where $[l, r] = \{l, l+1, \dotsc, r-1, r\}$.
    Therefore, at the end of the loop, we have $c = \rho$.
    This can easily be proved using induction.
    The base case is trivial as $[l, r] = [n]$ initially, and
    $\rho \in [n]$. The induction step follows from
    Proposition \ref{order}.
\end{proof}

We now have an algorithm to sample from the distribution
(\ref{srg_simplex_problem}). Unfortunately, 
while the search for $\rho$
only takes $O(\log{n})$ operations, 
the computation of $p^{*}$ alone takes $O(n)$ time,
while naively maintaining $(\norm{g_k^i})_{i=1}^{n}$ sorted
requires $O(n\log{n})$ operations and updating
$(\lambda(i))_{i=1}^{n}$ requires another $O(n)$ operations.
For large enough $n$, this can cause significant slowdown
even when the gradient evaluations are expensive.

\subsection{Tree-based implementation}
Our goal in this subsection is to design an efficient
algorithm that allows sampling
from (\ref{srg_simplex_problem}) in $O(\log{n})$ time.
To achieve this we need to:
\begin{itemize}
    \item Efficiently maintain 
    $(\norm{g_k^i})_{i=1}^{n}$ sorted from one iteration to
    the next while allowing for fast search.
    \item Quickly
    evaluate the $(\lambda(i))_{i=1}^{n}$ when we need them.
    \item Sample from $p^{*}$ without explicitly
    forming it.
\end{itemize}

To comply with all these requirements, we make use of an augmented
red-black tree $T$. We will assume that for a node $v$, its left child
$v.left$ has a smaller key than that of $v$, while the opposite is true
for its right child $v.right$. We will refer to its parent by $v.parent$.
If $v$ has no left (or right) child, we will assume that $v.left$
(or $v.right$) takes value $nil$.

The keys of the nodes of the tree $T$ are the gradient norms
$(\norm{g_k^i})_{i=1}^{n}$, while their values are the corresponding indices.
To simplify the presentation, we will assume that the keys are unique,
although everything here still works when they are not with some small modifications.
We require that each node $v$ of the tree $T$
maintains two more attributes:
\begin{itemize}
    \item $v.size$ which counts the number of nodes in the subtree whose root is $v$.
    \item $v.sum$ which stores the sum of the keys of the nodes in the subtree
    whose root is $v$.
\end{itemize}

We also make use of four methods that can be efficiently implemented using these
additional attributes:
\begin{itemize}
    \item $T.rank(v)$ which returns the rank 
    (position in the decreasing order of the tree) of the node $v$.
    \item $T.partial\_sum(v)$ which returns the sum of all keys larger than or
    equal to the key of the given node.
\end{itemize}

Before defining the last two methods, to each node $v$ we associate 
(conceptually only)
the interval $[T.partial\_sum(v) - v.key, T.partial\_sum(v))$.
The last two methods are:
\begin{itemize}
    \item $T.select\_rank(r)$ which returns the node with the $r^{th}$ largest key
    in the tree.
    \item $T.select\_sum(s)$ which returns the node whose associated 
    interval contains $s$.
\end{itemize}

Finally, and to allow fast access to the elements $\norm{g_k^i}$ through their
indices, we store the nodes of the tree in an array whose $i^{th}$ element
is the node whose value is $i$.
See Chapter 14 of \cite{cormen_introduction_2009} for a detailed description
of how to implement such a tree so that all the methods we require as well
as the maintenance of the tree run in $O(\log(n))$ operations.

With this notation in place, we can now state our $O(\log{n})$
method to sample from (\ref{srg_simplex_problem}) in Algorithm
\ref{sample}, which uses Algorithm \ref{solve} as a subroutine.

\begin{algorithm}[htbp]
    \caption{Solve}
    \begin{algorithmic}
        \STATE{\textbf{Input:} 
        tree $T$, lower bound $\eps > 0$
        }

        \STATE{$v \gets T.root$}
        \STATE{$r \gets v.right.size + 1$}
        \STATE{$s \gets v.right.sum + v.key$}
        \STATE{$c \gets 1 - (n - r) \eps$}

        \WHILE{$v \neq nil$}
            \IF{$v.key \geq \eps s / c$}
                \STATE{$w \gets v$}
                \STATE{$v \gets v.left$}
                \IF{$v \neq nil$}
                    \STATE $r \pluseq v.right.size + 1$
                    \STATE $s \pluseq v.right.sum + v.key$
                    \STATE $c = 1 - (n - r) \eps$
                \ENDIF
            \ELSE
                \STATE{$v \gets v.right$}
                \IF{$v \neq nil$}
                    \STATE $r \minuseq v.left.size + 1$
                    \STATE $s \minuseq v.left.sum + v.parent.key$
                    \STATE $c = 1 - (n - r) \eps$
                \ENDIF
            \ENDIF
        \ENDWHILE

        \STATE{\textbf{Output:} node $w$ whose rank is $\rho$}
    \end{algorithmic}
    \label{solve}
\end{algorithm}

\begin{algorithm}[htbp]
    \caption{Sample}
    \begin{algorithmic}
        \STATE{\textbf{Input:} 
        tree $T$, lower bound $\eps > 0$
        }

        \STATE{$w \gets Solve(T, \eps)$} \hfill \COMMENT{Using Algorithm \ref{solve}}
        \STATE{$\rho \gets T.rank(w)$}
        \STATE{$\lambda \gets T.partial\_sum(w)/(1 - (n-\rho)\eps)$}

        \STATE{sample $u$ uniformly from $[0, 1)$}
        \IF{$u < 1 - (n - \rho) \eps$}
            \STATE{$s \gets \lambda u$}
            \STATE{$v \gets T.select\_sum(s)$}
            \STATE{$i \gets v.value$}
            \STATE{$p^i \gets v.key / \lambda$}
        \ELSE
            \STATE{$u \minuseq 1 - (n - \rho) \eps$}
            \STATE{$r \gets n - \floor{u/\eps}$}
            \STATE{$v \gets T.select\_rank(r)$}
            \STATE{$i \gets v.value$}
            \STATE{$p^i \gets \eps$}
        \ENDIF

        \STATE{\textbf{Output:} the index $i$ and its probability $p^i$}
    \end{algorithmic}
    \label{sample}
\end{algorithm}

Similar to Algorithm \ref{naive_sample}, Algorithm \ref{sample}
first finds $\rho$, then uses Lemma \ref{solution} to sample
from (\ref{srg_simplex_problem}). 

To find $\rho$, Algorithm \ref{sample}
uses Algorithm \ref{solve}, which is only superficially different
from the $while$ loop in Algorithm \ref{naive_sample}. In particular,
it is not hard to show that the variable $r$ holds the rank
of the node $v$ (which corresponds to $m$ in Algorithm \ref{naive_sample}),
and the variable $s$ and $c$ satisfy $\lambda(r) = s/c$
(which corresponds to $\lambda(m)$ in Algorithm \ref{naive_sample}).
The correctness of Algorithm \ref{solve} therefore follows from that
of Algorithm \ref{naive_sample} which we proved in the last subsection.

Once $\rho$ is found, Algorithm \ref{sample} computes $\lambda = \lambda(p)$,
and samples the index $i$ from (\ref{srg_simplex_problem}) using
inverse transform sampling.
In particular, it partitions the unit interval into
$[0, 1 - (n-\rho)\eps)$ which is reserved to the $\rho$ largest
elements, and $[1 - (n-\rho)\eps, 1)$ which is reserved to the remaining
$(n - \rho)$ elements all of which have probability $\eps$. Using
the methods $T.select\_sum(s)$ and $T.select\_rank(r)$, it then picks
the right index and lazily evaluates the required probability.
Note that all the methods we use run in $O(\log{n})$ time, and therefore
the total complexity of sampling from (\ref{srg_simplex_problem})
and maintaining the tree $T$ is $O(\log{n})$.

\section{Convergence analysis}
\label{convergence}
In this section, we analyze the convergence of SRG,
and show that it outperforms SGD under our assumptions.
Two key constants are helpful in contrasting the convergence
of SRG and SGD.
Recall the definition of $\sigma^{2}(x_k, p_k)$ in (\ref{var}), and define:
\begin{align*}
    \sigma^2 &\coloneqq \sigma^{2}(\xstar, 1/n) = 
    \frac{1}{n} \sum_{i=1}^{n}\sqnorm{\nabla f_i(\xstar)} \\
    \sigma^2_{*} &\coloneqq \min_{p \in \Delta} \sigma^{2}(\xstar, p)
    = \frac{1}{n^2} \left(\sum_{i=1}^{n} \norm{\nabla f_i(\xstar)}\right)^2
\end{align*} 
It is well known that the convergence of SGD depends on $\sigma^2$
\cite{gower_sgd_2019}.
We here show that SRG is able to reduce this dependence to $\sigma_*^2$,
which can be up to $n$ times smaller than $\sigma^2$.

\subsection{Bounding the suboptimality  of $p_k$}
We start our analysis with the following Lemma, which is
a slightly improved version of Lemma 6 in \cite{borsos_online_2018}.
\begin{proposition}
    Let \(\{a_i\}_{i=1}^{n}\) be a non-negative 
    set of numbers and suppose that there exists an $i \in [n]$
    such that $a_i > 0$. Then:
    \begin{align*}
        \min_{p \in \Delta(\varepsilon)} \sum_{i=1}^{n} \frac{1}{p^i} a_i^2
        &\leq (1 + 2n\varepsilon) \min_{p \in \Delta}
        \sum_{i=1}^{n} \frac{1}{p^i} a_i^2 \\
        &= (1+2n\varepsilon) \left(\sum_{i=1}^{n}a_i\right)^2
    \end{align*}
    for all \(0 \leq \varepsilon \leq \frac{1}{2n}\).
    \label{subopt}
\end{proposition}
\begin{proof}
    By Lemma \ref{solution} we have:
    \begin{align*}
        \min_{p \in \Delta(\varepsilon)} \sum_{i=1}^{n} \frac{1}{p^i} a_i^2
        &= \lambda(\rho) \sum_{i=1}^{\rho}a_{\pi(i)} + 
        \sum_{i=\rho+1}^{n} \frac{a_{\pi(i)}^2}{\eps} \\
        &\leq \lambda(\rho) \sum_{i=1}^{\rho}a_{\pi(i)} + 
        a_{\pi(\rho + 1)} \sum_{i=\rho+1}^{n} \frac{a_{\pi(i)}}{\eps}
    \end{align*}
    Where the second line follows from the definition of $\pi$.
    By the reverse chain of implications in the proof of Proposition
    \ref{order} and taking $i=\rho$ we obtain 
    $a_{\pi(\rho + 1)} \leq \eps \lambda(\rho)$. Replacing we get:
    \begin{align*}
        \min_{p \in \Delta(\varepsilon)} \sum_{i=1}^{n} \frac{1}{p^i} a_i^2
        &\leq
        \lambda(\rho) \sum_{i=1}^{n}a_i \\
        &\leq
        \frac{1}{1-(n-\rho)\eps} \left(\sum_{i=1}^{n}a_i\right)^2 \\
        &\leq
        \frac{1}{1 - n\eps} \left(\sum_{i=1}^{n}a_i\right)^2 \\
    \end{align*}
    The result then follows from the inequality
    $\frac{1}{1-x} \leq 1 + 2x$ valid for $x \in [0, 1/2]$.
\end{proof}
Proposition \ref{subopt} is the main motivation for our choice of $p_k$
in Algorithm \ref{SRG}. If our goal was only to ensure
that $p_k > 0$, we could have first performed the optimization over
the entire simplex $\Delta$, which has a simple solution 
$p^i \propto \norm{g_k^i}$,
and then mixed this distribution
with the uniform distribution over $[n]$. This approach
however does not give us any suboptimality guarantees similar to Proposition
\ref{subopt}, which is crucial for the analysis as will soon become clear.

\subsection{Useful lemmas}
\label{useful_lemmas}
Before proceeding with the main results, let us first state a few useful lemmas.
We refer the reader to the Appendix for their proofs.
The first lemma studies the evolution of $(g_k^i)_{i=1}^{n}$.

\begin{lemma}
    Let $k \in \mathbb{N}$. Suppose $(g_k^i)_{i=1}^{n}$ evolves
    as in Algorithm \ref{SRG}, and assume that Assumptions 
    \ref{strongly-convex}, \ref{convex}, and \ref{smooth} hold.
    Taking expectation with respect to 
    $(i_k, b_k)$, conditional on $(i_t, b_t)_{t=0}^{k-1}$, we have:
    \begin{align*}
        &\E{\sum_{i=1}^{n} \sqnorm{g_{k+1}^i - \nabla f_i(x^*)}} \leq \\
        &\left(1 - \varepsilon_k\right) 
        \sum_{i=1}^{n} \sqnorm{g_k^i - \nabla f_i(x^*)}
        +
        2Ln\varepsilon_k \left[F(x_k) - F(x^*)\right]
    \end{align*}
    \label{tracking}
\end{lemma}

The second lemma is a bound on the second moment of the gradient estimator
used by SRG. 
\begin{lemma}
    Assume $\eps_k \in (0,1/2n]$. For all $\beta, \gamma, \delta, \eta > 0$ we have:
    \begin{align*}
        &\Exp{i_k \sim p_k}{\sqnorm{\frac{1}{n\pkik}\nabla f_{i_k}(x_k)}}
        \leq
        D_2 \frac{2L}{n\eps_k} \left[F(x_k) - F^*\right] + \\
        &D_1\frac{1}{n^2\eps_k} \sum_{i=1}^{n} \sqnorm{g_k^i - \nabla f_i(x^*)} + 
        D_3 (1 + 2n \varepsilon_k)\sigma_{*}^2
    \end{align*}
    where:
    \begin{align*}
        D_1 &\coloneqq (1 + \beta^{-1} + \delta) + 
        (1 + \gamma^{-1} + \delta^{-1}) (1 + \eta) \\
        D_2 &\coloneqq (1 + \beta + \gamma) \\
        D_3 &\coloneqq (1 + \gamma^{-1} + \delta^{-1})(1 + \eta^{-1})
    \end{align*}
    \label{second_moment}
\end{lemma}

\subsection{Main results}
To study the convergence of SRG, we use the following Lyapunov
function, which is the same (up to constants) as the one used to study
the convergence of SAGA in \cite{hofmann_variance_2015,defazio_simple_2016}.
\begin{align*}
    T^{k} &= T(x_k, (g_k^i)_{i=1}^{n}) \\
    &\coloneqq \frac{\alpha_k}{n\varepsilon_k} \frac{a}{L} 
    \sum_{i=1}^{n} \sqnorm{g_k^i - \nabla f_i(x^{*})}
    + \sqnorm{x_k - x^{*}}
\end{align*}
for the constant $a = 0.673$ that we set during the analysis.
Our main result is the following bound on the evolution
of $T^k$ along the steps of SRG which uses Lemmas 
\ref{tracking} and \ref{second_moment}.
All the proofs of this section can be found in
the Appendix.
\begin{theorem} 
    Suppose that Assumptions 
    \ref{strongly-convex}, \ref{convex} and \ref{smooth} hold,
    and that $(x_k, (g_k^i)_{i=1}^{n})$ evolves 
    according to Algorithm \ref{SRG}.
    Further, assume that for all \(k \in \mathbb{N}\):
    \begin{enumerate*}[label=(\roman*)]
        \item $\alpha_k/\eps_k$ is constant.
        \item $\eps_k \in (0, 1/2n]$.
        \item \(\alpha_k \leq n \eps_k/ 20L\).
    \end{enumerate*}
    Then for all \(k \in \mathbb{N}\):
    \begin{align*}
        \E{T^{k+1}} \leq (1-\alpha_k \zeta)\E{T^k} + (1 + 2n\varepsilon_k)2\alpha_k^2 \sigma_{*}^2
    \end{align*}
    where:
    \begin{align*}
        \zeta \coloneqq \min{\left\{L/n, \mu\right\}}
    \end{align*}
    \label{main}
\end{theorem}

The above theorem leads to the following convergence rate for SRG
with a constant step size.
\begin{corollary}
    Under the assumptions of Theorem \ref{main}, and using
    a constant lower bound 
    \(\eps_k = \eps \in (0, 1/2n]\)
    and a constant step size 
    \(\alpha_k = \alpha \leq n\eps/20L\),
    we have for any $k \in \mathbb{N}$:
    \begin{equation*}
        \E{T^k} \leq \left(1- \alpha \zeta\right)^{k} T^0 + (1+2n\varepsilon) \frac{2\alpha \sigma_*^2}{\zeta}
    \end{equation*}
    \label{constant}
\end{corollary}
Comparing this result with the standard convergence result
of SGD, for example Theorem 3.1 in \cite{gower_sgd_2019},
we see that they are of similar form. The advantage of this
result is the dependence on $\sigma^2_{*}$ instead 
of $\sigma^2$. This allows SRG to extend
the range of accuracies on which SGD converges linearly
on the one hand, and reduces the leading factor
in the complexity of SGD for high accuracies on the other.

For decreasing step sizes, we have the following result,
which again is similar to Theorem 3.2 in \cite{gower_sgd_2019},
but with $\sigma_*^2$ replacing $\sigma^2$.
\begin{corollary}
    Under the assumptions of Theorem \ref{main},
    and using the step sizes and lower bounds:
    \begin{align*}
        \alpha_k &\coloneqq 
        \frac{2(k + k_0) + 1}{\left[c + (k+k_0)(k+k_0+2)\right]\zeta} \\
        \eps_k &\coloneqq \frac{20L\alpha_k}{n}
    \end{align*}
    where:
    \begin{align*}
        k_0 &\coloneqq \frac{80L}{\zeta} - 2 \\
        c &\coloneqq \frac{40L}{\zeta}
    \end{align*}
    Then:
    \begin{equation*}
        \E{T^k} \leq O\left(\frac{T^{0} + \sigma_*^2 \log{k}}{k^2}\right) + O\left(\frac{\sigma_*^2}{k}\right)
    \end{equation*}
    \label{decreasing}
\end{corollary}

\begin{figure*}
   \centering
   \includegraphics[width=1.0\textwidth]{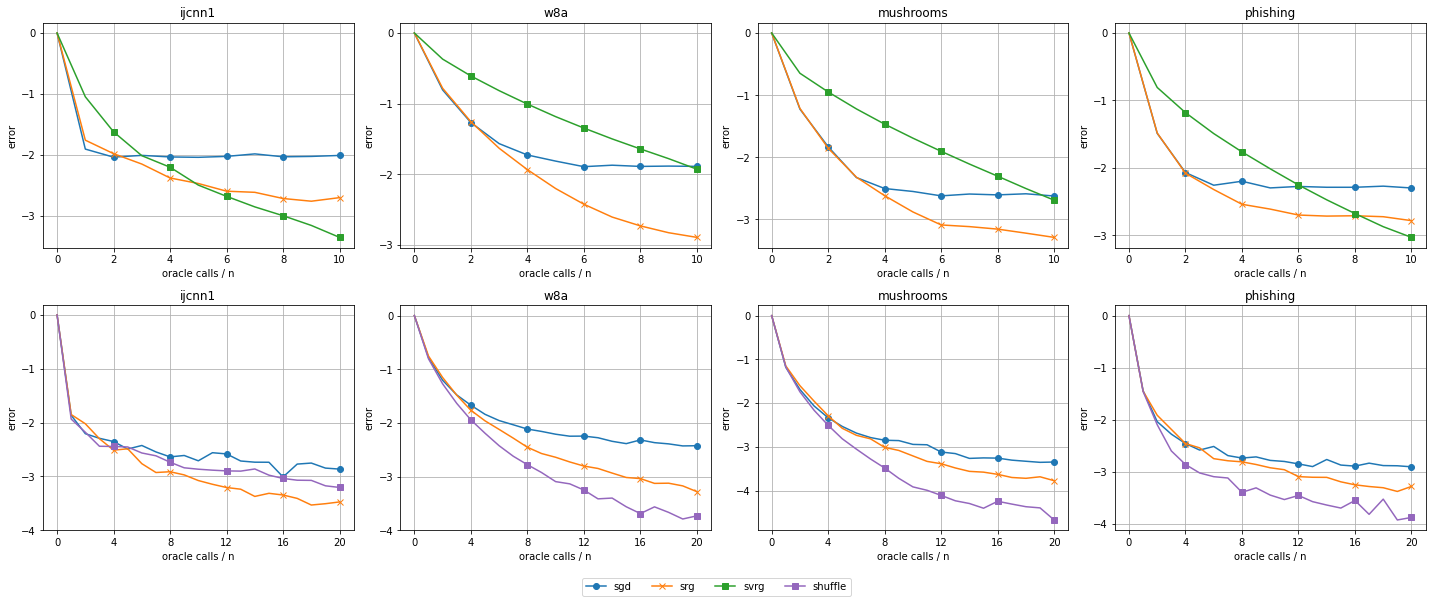}
   \caption{Comparison of the evolution of the relative error for different optimizers on $\ell_2$-regularized
   logistic regression problems using the datasets \emph{ijcnn1, w8a, mushrooms} and \emph{phishing}.
   The top row compares SGD (blue), SRG (orange), and SVRG (green), all with constant step size.
   The bottom row compares SGD (blue), SRG (orange), and SGD with shuffling (purple), all using
   decreasing step sizes.}
   \label{empirical}
\end{figure*}

\section{Experiments}
\label{experiments}
In this section, we experiment with SRG
on $\ell_2$-regularized logistic regression problems.
In this case, the functions $f_i$ are given by:
\begin{equation*}
    f_i(x) \coloneqq \log{(1 + \exp{(-y_i a_i^T x)})} + \frac{\mu}{2}\sqnorm{x}
\end{equation*}
where $y_i \in \{0, 1\}$ is the label of data point $a_i \in \Rd$. It is not hard
to show that each $f_i$ is convex and $L_i = 0.25\sqnorm{a_i} + \mu$ smooth.
Their average $F$ is also $\mu$-strongly convex.

We experiment with four datasets from LIBSVM \cite{chang_libsvm_2011}:
\emph{ijcnn1, w8a, mushrooms,} and \emph{phishing}.
We start by normalizing the data points $(a_i)_{i=1}^{n}$ so that $\norm{a_i} = 1$ for all $i \in [n]$.
This makes all the $f_i$ $L$-smooth with $L = 0.25 +\mu$. As is standard for regularized logistic 
regression, we take $\mu = \frac{1}{n}$.
We conduct two sets of experiments: one set using a constant step size, and another using
decreasing step sizes. In both cases we evaluate the performance of the algorithms
by tracking the relative error $\frac{\sqnorm{x_k - x^{*}}}{\sqnorm{x_0 - x^{*}}}$.
For all experiments, we sample the indices without replacement using a batch size of $128$.
Finally, we use the same step size across all algorithms for each dataset
to fairly compare their performance. We give more details on the step sizes we
use in the supplementary material.

The results of the experiments with constant step-size can be seen on the top row of
Figure \ref{empirical}. We compare SRG with SGD on the one hand and SVRG on the other.
We observe that SRG consistently outperforms SGD, reaching solutions that are about an order
of magnitude more accurate. We also notice that SRG extends the range of accuracies
on which SGD dominates SVRG. Eventually, SVRG outperforms both SGD and SRG, but our results show that SRG allows 
reaching solutions of moderate accuracy more quickly.

The bottom row of Figure \ref{empirical} shows the results of the experiments where we use decreasing
step-sizes. Here we compare SRG with SGD as well as SGD with random shuffling. 
SGD with random shuffling has attracted a lot of attention recently, and many analyses
have been developed that show its superior performance, see for example \cite{ahn_sgd_2020}.
Our results confirm this as SGD with shuffling outperforms both SGD and SRG. We still see however
that SRG converges slightly faster than SGD, as our analysis suggests.

While the empirical results we present here are positive but not impressive, we would like to point out
that the improvement of SRG over SGD depends directly on the ratio $r = \sigma^2/\sigma^2_{*}$. In particular,
if $\sigma_{*}^2 = \sigma^2$, then we should not expect any improvement. For the experiments in
Figure \ref{empirical}, $r$ is in the range $[5,10]$. One can construct artificial examples
in which this ratio is very large, which we do in the supplementary material, leading to significant improvements.
It is however still unclear how often such cases are encountered in practice. We hope
to explore this question more in future work.

\newpage
\bibliography{srg}
\bibliographystyle{icml2021}

\onecolumn
\appendix

\icmltitle{Supplementary material}
\section{Standard results}
\label{stand_results}
\begin{lemma}
    Under Assumptions \ref{convex} and \ref{smooth},
    we have:
    \begin{equation*}
        \frac{1}{n}\sum_{i=1}^{n} \sqnorm{\nabla f_i(x) - \nabla f_i(\xstar)}
        \leq 2 L \left[F(x) - F(\xstar)\right]
    \end{equation*}
    \label{grad_norm_bound}
\end{lemma}
\begin{proof}
    By Assumptions \ref{convex} and \ref{smooth},
    and Theorem 2.1.5 in \cite{nesterov_introductory_2004}
    we have for all $i \in [n]$:
    \begin{equation*}
        f_i(x) \leq f_i(\xstar) + \inp{\nabla f_i(\xstar)}{x - \xstar}
        + \frac{1}{2L}\sqnorm{\nabla f_i(x) - \nabla f_i(\xstar)}
    \end{equation*}
    Rearranging gives:
    \begin{equation*}
        \sqnorm{\nabla f_i(x) - \nabla f_i(\xstar)}
        \leq 2L\left[f_i(x) - f_i(\xstar) - 
        \inp{\nabla f_i(\xstar)}{x - \xstar}\right]
    \end{equation*}
    Averaging over $i \in [n]$ and noticing that
    $\frac{1}{n}\sum_{i=1}^{n}\nabla f_i(\xstar) = \nabla F(\xstar) = 0$, we
    get the result.
\end{proof}

\begin{proposition} (Young's inequality, Peter-Paul inequality)
    Let $a, b, c \in \Rd$. Then for all $\beta, \gamma, \delta > 0$
    we have:
    \begin{align*}
        \sqnorm{a + b} &\leq (1 + \beta) \sqnorm{a} + 
        (1 + \beta^{-1}) \sqnorm{b} \\
        \sqnorm{a + b + c} &\leq
        (1 + \beta + \gamma) \sqnorm{a} +
        (1 + \beta^{-1} + \delta) \sqnorm{b} + 
        (1 + \gamma^{-1} + \delta^{-1}) \sqnorm{c}
    \end{align*}
    \label{youngs_inequality}
\end{proposition}
\begin{proof}
    We prove the first statement. The second follows from a similar
    argument. We have for $x, y \in \R$:
    \begin{align*}
        & \beta x^2 - 2xy + \beta^{-1}y^2 
        = \left(\sqrt{\beta}x - \sqrt{\beta^{-1}}y\right)^2
        \geq 0
        \Rightarrow
        2xy \leq \beta x^2 + \beta^{-1} y^2 \\
        & \beta x^2 + 2xy + \beta^{-1}y^2 
        = \left(\sqrt{\beta}x + \sqrt{\beta^{-1}}y\right)^2
        \geq 0
        \Rightarrow
        -2xy \leq \beta x^2 + \beta^{-1} y^2
    \end{align*}
    Therefore:
    \begin{equation*}
        2\abs{xy} \leq \beta x^2 + \beta^{-1} y^2
    \end{equation*}
    Now by Cauchy-Schwarz inequality:
    \begin{align*}
        \sqnorm{a + b} &= \sqnorm{a} + 2\inp{a}{b} + \sqnorm{b} \\
        &\leq \sqnorm{a} + 2\abs{\inp{a}{b}} + \sqnorm{b} \\
        &\leq \sqnorm{a} + 2\norm{a}\norm{b} + \sqnorm{b} \\
        &\leq (1 + \beta) \sqnorm{a} + (1 + \beta^{-1}) \sqnorm{b}
    \end{align*}
\end{proof}

\section{Missing proofs}
\label{missing_proofs}
\subsection{Proof of Lemma \ref{tracking}}
\begin{proof}
    Taking expectation with respect to $(i_k, b_k)$ conditional
    on $(i_t, b_t)_{t=1}^{k-1}$ we have:
    \begin{align*}
        &\E{\sum_{i=1}^{n} \sqnorm{g_{k+1}^{i} - \nabla f_{i}(x^{*})}}\\
        &= \sum_{j=1}^{n} \mathbb{P}(i_k = j)
        \biggl[
            \mathbb{P}(b_k = 0 \mid i_k = j) \left(
                \sum_{i=1}^{n} \sqnorm{g_{k}^{i} - \nabla f_{i}(x^{*})}
                \right)
        \\
        &\quad +
        \mathbb{P}(b_k = 1 \mid i_k = j)
        \left(
            \sum_{\substack{i=1 \\ i \neq j}}^{n} 
            \sqnorm{g_k^i - \nabla f_i(x^{*})} + 
            \sqnorm{\nabla f_j(x_k) - \nabla f_j(x^{*})}
        \right)
        \biggl] \\
        &= \sum_{j=1}^{n} p_k^j 
        \left[
            \left(1 - \frac{\varepsilon_k}{p_k^j}\right)
            \sum_{i=1}^{n} \sqnorm{g_{k}^{i} - \nabla f_{i}(x^{*})}
            +
            \frac{\varepsilon_k}{p_k^j} \left(
                \sum_{\substack{i=1 \\ i \neq j}}^{n} 
                \sqnorm{g_k^i - \nabla f_i(x^{*})} + 
                \sqnorm{\nabla f_j(x_k) - \nabla f_j(x^{*})}
            \right)
        \right] \\
        &= \left(1 - \varepsilon_k\right) \sum_{i=1}^{n} 
        \sqnorm{g_{k}^{i} - \nabla f_{i}(x^{*})} + 
        \varepsilon_k \sum_{i=1}^{n} \sqnorm{\nabla f_i(x_k) - \nabla f_i(x^{*})} \\
        &\leq \left(1 - \varepsilon_k\right) \sum_{i=1}^{n} 
        \sqnorm{g_{k}^{i} - \nabla f_{i}(x^{*})} + 
        2 L n\varepsilon_k \left[F(x_k) - F(x^{*})\right]
    \end{align*}
    where the second and third lines follow from the update
    of Algorithm \ref{SRG}, and the last line follows from Lemma 
    \ref{grad_norm_bound}.
\end{proof}

\subsection{Proof of Lemma \ref{second_moment}}
\begin{proof}
    Taking expectation with respect to $(i_k, b_k)$ conditional
    on $(i_t, b_t)_{t=1}^{k-1}$ we have:
    \begin{equation*}
        \E{\sqnorm{\frac{1}{n\pkik} \nabla f_{i_k}(x_k)}}
        = \frac{1}{n^2} \sum_{i=1}^{n} \frac{1}{p_k^i} 
        \sqnorm{\nabla f_i(x_k)}
    \end{equation*}
    Now by Proposition \ref{youngs_inequality}:
    \begin{align*}
        \sum_{i=1}^{n} \frac{1}{p_k^i} \sqnorm{\nabla f_{i}(x_k)}
        &= 
        \sum_{i=1}^{n} \frac{1}{p_k^i} 
        \sqnorm{
            \nabla f_{i}(x_k) - \nabla f_{i}(x^{*}) + 
            \nabla f_{i}(x^{*}) - g_k^{i} +
            g_k^i
        } \\
        &\leq 
        (1+\beta+\gamma) 
        \sum_{i=1}^{n} \frac{1}{p_k^i} 
        \sqnorm{\nabla f_{i}(x_k) - \nabla f_{i}(x^{*})} \\
        &\quad +
        (1+\beta^{-1} + \delta) 
        \sum_{i=1}^{n} 
        \frac{1}{p_k^i}\sqnorm{g_k^i - \nabla f_{i}(x^{*})} \\
        &\quad + 
        (1 + \gamma^{-1} + \delta^{-1}) 
        \sum_{i=1}^{n} 
        \frac{1}{p_k^{i}} \sqnorm{g_k^i} \\
    \end{align*}
    Let us bound each of the three terms. The first can be bound
    using $p_k \geq \varepsilon_k$ and Lemma \ref{grad_norm_bound}:
    \begin{align*}
        \sum_{i=1}^{n} 
        \frac{1}{p_k^i} 
        \sqnorm{\nabla f_{i}(x_k) - \nabla f_{i}(x^{*})}
        &\leq
        \frac{1}{\varepsilon_k} 
        \sum_{i=1}^{n} 
        \sqnorm{\nabla f_i(x_k) - \nabla f_i(x^*)}
        \leq 
        \frac{2Ln}{\varepsilon_k}\left[F(x_k) - F(x^*)\right]
    \end{align*}
    The second is easily bound using $p_k \geq \varepsilon_k$:
    \begin{align*}
        \sum_{i=1}^{n} 
        \frac{1}{p_k^i}
        \sqnorm{g_k^i - \nabla f_i(x^*)}
        \leq 
        \frac{1}{\varepsilon_k}
        \sum_{i=1}^{n} 
        \sqnorm{g_k^i - \nabla f_i(x^*)}
    \end{align*}
    Finally, the third term can be bound as:
    \begin{align*}
        &\sum_{i=1}^{n} \frac{1}{p_k^i} \sqnorm{g_k^i} \\
        &\leq 
        (1 + 2n\varepsilon_k)
        \left( \sum_{i=1}^{n} \norm{g_k^i} \right)^2 \\
        &\leq
        (1 + 2n\varepsilon_k)
        \left( 
        \sum_{i=1}^{n}
        \norm{g_k^i - \nabla f_i(x^*)}
        +
        \sum_{i=1}^{n}
        \norm{\nabla f_i(x^{*})}
        \right)^2 \\
        &\leq
        (1+2n\varepsilon_k)
        (1+\eta)
        \left(\sum_{i=1}^{n} \norm{g_k^i - \nabla f_i(x_k)} \right)^2
        +
        (1+2n\varepsilon_k)
        (1 + \eta^{-1})
        \left(\sum_{i=1}^{n} \norm{\nabla f_i(x^*)} \right)^2 \\
        &\leq
        (1+2n\varepsilon_k)
        (1 + \eta) n
        \sum_{i=1}^{n} \sqnorm{g_k^i - \nabla f_i(x^*)}
        + 
        (1+2n\varepsilon_k)
        (1 + \eta^{-1})
        n^2 \sigma_{*}^2 \\
        &\leq
        2n (1+\eta) \sum_{i=1}^{n} \sqnorm{g_k^i - \nabla f_i(x^*)}
        +
        (1+2n \varepsilon_k)
        (1 + \eta^{-1})
        n^2 \sigma_{*}^2 \\
        &\leq
        (1+\eta) \frac{1}{\varepsilon_k} 
        \sum_{i=1}^{n} \sqnorm{g_k^i - \nabla f_i(x^*)}
        +
        (1+ \eta^{-1}) (1+2n \varepsilon_k) n^2\sigma_{*}^2
    \end{align*}
    Where the second line follows from Proposition \ref{subopt},
    the third from the triangle inequality,
    and the fourth by Proposition \ref{youngs_inequality}.
    The fifth line follows from the definition of $\sigma^2_*$ in
    section \ref{convergence} and an application of Cauchy-Schwarz
    inequality. Let $v$ be the vector whose $i^{th}$ component
    is $\norm{g_k^i - \nabla f_i(\xstar)}$ and let
    $1_n$ be the vector of ones. Then:
    \begin{equation*}
        \left(\sum_{i=1}^{n} \norm{g_k^i - \nabla f_i(\xstar)}\right)^2 
        = \abs{\inp{1_n}{v}}^2 \leq \sqnorm{1_n} \sqnorm{v}
        = n \sum_{i=1}^{n} \sqnorm{g_k^i - \nabla f_i(\xstar)}
    \end{equation*}
    Finally, lines six and seven follow from the 
    inequality $\varepsilon_k \leq 1/2n$.
    Combining the three bounds yields the result.
\end{proof}
\subsection{Proof of Theorem \ref{main}}
\begin{proof}
    Recall that we are studying the one-step evolution of the
    Lyapunov function:
    \begin{equation*}
        T^k = T(x_k, (g_k^i)_{i=1}^{n})
        \coloneqq \frac{\alpha_k}{n\varepsilon_k} \frac{a}{L} 
        \sum_{i=1}^{n} \sqnorm{g_k^i - \nabla f_i(x^{*})}
        + \sqnorm{x_k - x^{*}}
    \end{equation*}
    and that we are assuming
    \begin{enumerate*}[label=(\roman*)]
        \item $\alpha_k/\eps_k$ is constant.
        \item $\eps_k \in (0, 1/2n]$.
        \item \(\alpha_k \leq n \eps_k/ 20L\).
    \end{enumerate*}

    All the expectations in this proof are 
    with respect to $(i_k, b_k)$ and are
    conditional  on \((i_t, b_t)_{t=0}^{k-1}\).
    Since $\alpha_k/\eps_k$ is constant, Lemma \ref{tracking}
    immediately gives us a bound on the first term
    of $\E{T^{k+1}}$.


    The second term of $\E{T^{k+1}}$ expands as:
    \begin{align*}
        \E{\sqnorm{x_{k+1} - x^{*}}} &= 
        \E{\sqnorm{x_k - \alpha_k \frac{1}{n p_k^{i_k}} 
        \nabla f_{i_k}(x_k) - x^* }} \\
        &= \sqnorm{x_k - x^*} - 
        2\alpha_k \inp{\E{\frac{1}{np_k^{i_k}} 
        \nabla f_{i_k}(x_k)}}{x_k - x^*}
        +
        \alpha_k^2 \E{\sqnorm{\frac{1}{np_k^{i_k}} \nabla f_{i_k}(x_k)}} \\
        &= \sqnorm{x_k - x^*} - 2\alpha_k \inp{\nabla F(x_k)}{x_k - x^*}
        + \alpha_k^2 \E{\sqnorm{\frac{1}{np_k^{i_k}} \nabla f_{i_k}(x_k)}} \\
        &\leq (1 - \alpha_k \mu)\norm{x_k - x^*} - 
        2\alpha_k \left[F(x_k) - F(x^*)\right] + 
        \alpha_k^2 \E{\sqnorm{\frac{1}{np_k^{i_k}} \nabla f_{i_k}(x_k)}}
    \end{align*}
    where in the last line we use Assumption \ref{strongly-convex} 
    (strong-convexity of $F$).
    Since we are assuming $\eps_k \in (0, 1/2n]$,
    we can apply Lemma \ref{second_moment} to bound the last term above.
    Combining the resulting bound with the one on the first
    term we get:
    \begin{align*}
        \E{T^{k+1}} &\leq 
        \left(1 - \eps_k + \frac{D_1 \alpha_k L}{na}\right)
        \frac{\alpha_k}{n\eps_k}\frac{a}{L}\sum_{i=1}^{n}
        \sqnorm{g_k^i - \nabla f_i(x^*)}
        + \left(1 - \alpha_k \mu\right) \sqnorm{x_k - x^*} \\
        &\quad + D_3 (1 + 2n\varepsilon_k) \alpha_k^2 \sigma_*^2 \\
        &\quad + \frac{2\alpha_k L}{n\varepsilon_k}
        \left(D_2 \alpha_k + 
        \frac{an\varepsilon_k}{L}
        -
        \frac{n\varepsilon_k}{L}
        \right) \left[F(x_k) - F(x^*)\right]
    \end{align*}

    To ensure that the last parenthesis is not positive,
    we need:
    \begin{equation}
        \alpha_k \leq \frac{(1-a)}{D_2}\frac{n\eps_k}{L}
        \label{step_size_cond}
    \end{equation}
    Assuming $\alpha_k$ satisfies this condition,
    and replacing in the first parenthesis we get:
    \begin{equation*}
        1 - \eps_k + \frac{D_1\alpha_kL}{na}
        \leq 1 - \alpha_k \left(\frac{D_2}{1-a} - \frac{D_1}{a}\right) \frac{L}{n}
    \end{equation*}
    Now:
    \begin{align*}
        \min\left\{\left(\frac{D_2}{1-a} - \frac{D_1}{a}\right) \frac{L}{n}, \mu\right\}
        \leq \min\left\{\left(\frac{D_2}{1-a} - \frac{D_1}{a}\right), 1\right\} \zeta
    \end{align*}
    where:
    \begin{equation*}
        \zeta \coloneqq \min\left\{\frac{L}{n}, \mu\right\}
    \end{equation*}
    So that, assuming (\ref{step_size_cond}) holds, 
    our upper bound on $\E{T^{k+1}}$ is itself upper bounded by:
    \begin{align*}
        \E{T^{k+1}} \leq \left(1 - \alpha_k  \min\left\{\left(\frac{D_2}{1-a} - \frac{D_1}{a}\right), 1\right\} \zeta \right)
        T^{k} + D_3 (1 + 2n\eps_k) \alpha_k^2 \sigma_{*}^2
    \end{align*}

    It remains to choose the parameters $\beta, \gamma, \delta, \eta > 0$,
    and the parameter $a > 0$ so as to
    minimize this upper bound while maximizing the step size
    in (\ref{step_size_cond}).
    First however, note that we have the constraint
    $a < 1$ so that the step size can be allowed to be positive
    in (\ref{step_size_cond}). Furthermore, we choose
    to impose $D_3 = 2$ since this is also the leading
    factor in the analogous term in the analysis of SGD
    \cite{gower_sgd_2019}.
 
    These considerations lead us to the following 
    constrained optimization problem:
    \begin{align*}
        \max_{a, \beta, \gamma, \delta, \eta} \quad &\frac{(1-a)}{D_2} \\
        \text{subject to:} \quad & \frac{D_2}{1-a} - \frac{D_1}{a} \geq 1 \\
        & D_3 \leq 2 \\
        & 0 < a < 1 \\
        & \beta, \gamma, \delta, \eta > 0 
    \end{align*}
    which we solve numerically to find the feasible point:
    \begin{equation*}
        a = 0.673, \
        \beta = 1.028, \
        \gamma = 4.084, \ 
        \delta = 3.973, \ 
        \eta = 2.980
    \end{equation*}
    With this choice of parameters, we get:
    \begin{equation*}
        \E{T^{k+1}} \leq \left(1 - \alpha_k \zeta \right)
        T^{k} + (1 + 2n\eps_k) 2 \alpha_k^2 \sigma_{*}^2
    \end{equation*}
    under the condition:
    \begin{equation*}
        \alpha_k \leq \frac{1}{20}\frac{n\eps_k}{L}
    \end{equation*}
    which ensures that (\ref{step_size_cond}) holds.
\end{proof}
\subsection{Proof of Corollary \ref{constant}}
\begin{proof}
    Under the assumptions of the corollary,
    the conditions of Theorem \ref{main} are
    satisfied for all $k \in \mathbb{N}$.
    Starting from $\E{T^k}$ and
    repeatedly applying Theorem \ref{main} we get:
    \begin{align*}
        \E{T^{k}} &\leq (1- \alpha \zeta)^k T^{0} + 
        (1 + 2n\varepsilon)2\alpha^2\sigma_*^2
        \sum_{t=0}^{k-1} (1- \alpha \zeta)^{t} \\
        &\leq (1 - \alpha \zeta)^k T^{0} + 
        (1 + 2n\varepsilon)2\alpha^2\sigma_*^2
        \sum_{t=0}^{\infty} (1- \alpha \zeta)^{t} \\
        &= (1 - \alpha \rho)^k T^{0} + 
        (1 + 2n\varepsilon)\frac{2\alpha\sigma_*^2}{\zeta}
    \end{align*}
\end{proof}

\subsection{Proof of Corollary \ref{decreasing}}
\begin{proof}
    The choices of $\eps_k$, $\alpha_k$, and of the constants
    $c$ and $k_0$ are such that:
    \begin{enumerate*}[label=(\roman*)]
        \item $\alpha_k/\eps_k$ is constant.
        \item $\alpha_k \leq n\eps_k/20L$ for all $k \in \mathbb{N}$.
        \item $\alpha_k$, and therefore $\eps_k$, is decreasing for all $k \in \mathbb{N}$.
        \item $\eps_0 = \frac{1}{2n}$.
    \end{enumerate*}
    Therefore, the conditions of Theorem \ref{main} hold
    for all $k \in \mathbb{N}$.
    Fix \(k \in \mathbb{N}\). We have by Theorem \ref{main}
    for any \(t \in [k-1] \cup \{0\}\):
    \begin{align*}
        \E{T^{t+1}} &\leq 
        (1-\alpha_t\zeta) \E{T^t} + 
        (1 + 2n\varepsilon_t)2\alpha_t^2 \sigma_{*}^2 \\[1em]
        &= \frac{c + (t+k_0-1)(t+k_0+1)}{c + (t+k_0)(t+k_0+2)} \E{T^{t}}
        + (1 + 2n\varepsilon_t)2\alpha_t^2 \sigma_*^2
    \end{align*}
    multiplying both sides by \(\left[c + (t+k_0)(t+k_0+2)\right]\) and
    noticing that:
    \begin{equation*}
        \left[c + (t+k_0)(t+k_0+2)\right]\alpha_t^2 = 
        \frac{\left[2 (t + k_0) + 1\right]^2}
        {\left[c + (t+k_0)(t+k_0+2)\right]} \leq 4
    \end{equation*}
    we get:
    \begin{align*}
        \left[c + (t+k_0)(t+k_0+2)\right]\E{T^{t+1}} - 
        \left[c + (t+k_0-1)(t+k_0+1)\right]\E{T^{t}}
        \leq (1+2n\varepsilon_t)8\sigma_*^2
    \end{align*}
    summing the \(k\) inequalities for \(t \in [k-1] \cup \{0\}\) and noticing
    that the left side is a telescoping sum we obtain:
    \begin{align*}
        \left[c + (k+k_0)(k+k_0+2)\right]\E{T^{k}}
        -
        \left[c + (k_0 - 1)(k_0 + 1)\right] T^{0}
        \leq 
        8k\sigma_{*}^2 + 16n\sigma_{*}^2 \sum_{t=0}^{k-1}\varepsilon_t
    \end{align*}
    Let us now bound the last term:
    \begin{align*}
        16 n \sigma_{*}^2 \sum_{t=0}^{k-1} \varepsilon_t
        &= 320L\sigma_*^2\sum_{t=0}^{k-1} \alpha_t \\[1em]
        &= \frac{320L \sigma_*^2}{\zeta}
        \sum_{t=0}^{k-1} \frac{2 (t + k_0) + 1}{\left[c + (t+k_0)(t+k_0+2)\right]}\\[1em]
        &\leq \frac{320L \sigma_*^2}{\zeta} 
        \int_{-1}^{k-1} \frac{2(x + k_0) + 1}{\left[c + (x+k_0)(x+k_0+2)\right]} \, dx\\[1em]
        &\leq \frac{320L \sigma_*^2}{\zeta} 
        \int_{-1}^{k-1} \frac{2(x + k_0 + 1)}{\left[c + (x+k_0)(x+k_0+2)\right]} \, dx\\[1em]
        &= \frac{320L \sigma_*^2}{\zeta} 
        \log{\frac{\left[c + (k+k_0-1)(k+k_0+1)\right]}{\left[c + (k_0-1)(k_0+1)\right]}}
    \end{align*}
    where the first line follows from the definition of $\eps_t$,
    and the third line follows from the fact that the integrand is
    decreasing over the domain of integration.
    Using this bound and rearranging:
    \begin{align*}
        \E{T^k} &\leq 
        \frac{\left[c + (k_0-1)(k_0+1)\right]}
        {\left[c + (k+k_0)(k+k_0+2)\right]}
        T^{0}
        +
        \frac{8k\sigma_*^2}{\left[c + (k+k_0)(k+k_0+2)\right]}
        \\[1em]
        &\quad +
        \frac{320L \sigma_*^2}{\zeta} 
        \frac{1}{\left[c + (k+k_0)(k+k_0+2)\right]}
        \log{\frac{\left[c + (k+k_0-1)(k+k_0+1)\right]}{\left[c + (k_0-1)(k_0+1)\right]}}
        \\
        &= O\left(\frac{T^{0} + \sigma_*^2 \log{k}}{k^2}\right) + O\left(\frac{\sigma_*^2}{k}\right)
    \end{align*}
\end{proof}

\section{Details of experiments}
\label{exp_details}
In this section we give more details on the experiments
presented in section \ref{experiments}.

\paragraph{Batch size:} Our analysis is for the case where the batch size is equal to 1.
We would have therefore ideally liked to run the experiments
with a unit batch size. This is however extremely slow,
so we used instead a batch size of $m = 128$ sampled without
replacement.

\paragraph{Step sizes:} To make a fair comparison between the algorithms, we use
the same step size sequence across them for a given dataset.
Define:
\begin{align*}
    \mathcal{L} = \frac{n - m}{m(n-1)} L_{max} + \frac{n(m-1)}{m(n-1)} L
\end{align*}
where $L_{max} = \max_{i \in [n]} L_i$ where $L_i$ is the smoothness constant
of $f_i$, $L$ is the smoothness constant of $F$, and $m$ is the batch size.

For the constant step size experiments, we use SGD as a reference,
and use the maximum step size allowable in its analysis in
\cite{gower_sgd_2019}, which is given by:
\begin{equation*}
    \alpha = \frac{1}{2\mathcal{L}}
\end{equation*}

Similarly, for the decreasing step size experiments, 
we use the step sizes:
\begin{equation*}
    \alpha_k = \frac{2(k + k_0) + 1}
    {\left[c + (k+k_0)(k+k_0+2)\right]\mu}
\end{equation*}
where:
\begin{align*}
    k_0 &= \frac{4\mathcal{L}}{\mu} - 2 \\
    c &= \frac{2\mathcal{L}}{\mu}
\end{align*}

\paragraph{Algorithm specific parameters:}For SRG, we initialize $g_0^i = 0$ for all $i \in [n]$,
and we always update $g_k^{i_k}$, i.e. we do not sample
a Bernoulli to decide whether the update occurs or not,
as we believe that the added Bernoulli is only an artifact of the
analysis.
Furthermore, as the indices are sampled without replacement,
the naive importance sampling estimator is biased. 
We use the mini-batch estimator introduced in
\cite{el_hanchi_adaptive_2020} which corrects for this bias.
For the constant step size experiments, we set $\eps_k = \eps = 1/2n$, 
while for the ones using decreasing step sizes,
we set $\eps_k = \mathcal{L}\alpha_k/n$

For SVRG, we use its loopless version \cite{hofmann_variance_2015}
to avoid the stairlike behavior of the convergence plot
of SVRG. We use a full update probability of $q = m/n$ so that
the expected number of gradient evaluations per iteration
is three times that of SGD, the same ratio that we would expect
from running standard SVRG and SGD with a unit batch size.

For SGD with random shuffling, we sample a fresh permutation
at the beginning of each epoch.

\paragraph{Figure:} Figure \ref{empirical} shows the evolution of the log relative
error:
\begin{equation*}
    \log_{10}\left(\frac{\sqnorm{x_k - \xstar}}{\sqnorm{x_0 - \xstar}}\right)
\end{equation*}
For each algorithm and each dataset, we run the algorithm
ten times and plot the average of the results.

\section{Synthetic experiment}
\label{synth_exp}
In this section, we compare SRG and SGD on a synthetic dataset.
Our theoretical analysis suggests that the
improvement of SRG over SGD is proportional to the
ratio $r = \sigma^2/\sigma_*^2$.
Here we construct an artificial example in which this ratio 
is large ($r \approx 50$), compared to $r \in [5,10]$ for the experiments
in section \ref{experiments}.

We do this by generating a synthetic dataset as follows:
\begin{itemize}
    \item generate a $1000 \times 10$ features matrix $A$ randomly with $a_{ij} \sim \mathcal{N}(0,1)$.
    \item generate a weight vector $w$ randomly with $w_i \sim \mathcal{N}(0,1)$.
    \item generate the target values randomly as:
    \begin{equation*}
        y_i = w^T a_i + \eps_i
    \end{equation*} 
    where each $\eps_i$ is an independent 
    standard Cauchy random variable.
\end{itemize}
We then fit a linear regression model with
the mean squared error loss using SGD and SRG.
We use a batch size of $1$, and a constant step size
$\alpha = 1/2L$ for both. For SRG we use $\eps_k = \eps = 1/2n$,
and initialize $g_0^i = 0$.
The results are displayed
in Figure \ref{synthetic}.
We ran each algorithm a hundred times and plotted the average
of the results.
The left plot shows the evolution of the log relative error,
while the right one shows the evolution of the log of the
variance of the gradient estimator $\sigma^2(x_k, 1/n)$ and
$\sigma^2(x_k, p_k)$ respectively for SGD and SRG.
We see that SRG reaches solution that 
are approximately two orders of magnitude more accurate
than SGD, which is what we expect from the value of the ratio
$r\approx 50$. As pointed out in section \ref{experiments},
it is not clear how often such large values for the ratio $r$
are encountered in practice, but our experiments confirm that 
substantial gains can be achieved if it is large enough.

\begin{figure*}
    \centering
    \includegraphics[width=1.0\textwidth]{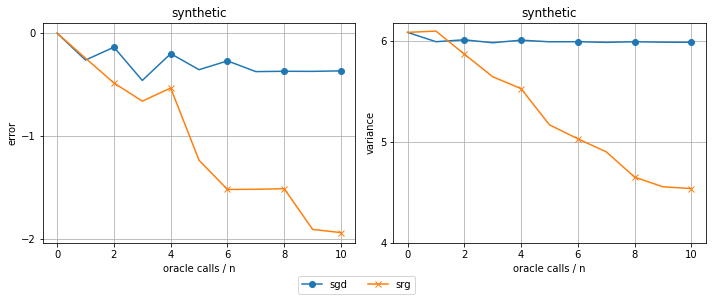}
    \caption{Comparison of the evolution of the relative error (left)
    and the variance of the gradient estimator (right) for SGD
    and SRG applied to the synthetic dataset described in 
    section \ref{synth_exp}.}
    \label{synthetic}
 \end{figure*}


\end{document}